\documentclass[11pt]{amsart}

\usepackage{amsmath}
\usepackage{amssymb}
\usepackage{amscd}
\usepackage{color}

\begin{document}

\title{Prime Structures in a Morita Context}

\author[M. B. Calci]{Mete Burak Calci}
\address{Mete Burak Calci, Department of Mathematics, Ankara University,  Turkey}
\email{mburakcalci@gmail.com}

\author[S. Halicioglu]{Sait Halicioglu}
\address{Sait Halicioglu, Department of Mathematics, Ankara University, Turkey}
\email{halici@ankara.edu.tr}

\author[A. Harmanci]{Abdullah Harmanci}
\address{Abdullah Harmanci, Department of Mathematics, Hacettepe University,~Turkey}
\email{harmanci@hacettepe.edu.tr}

\author[B. Ungor]{Burcu Ungor}
\address{Burcu Ungor, Department of Mathematics, Ankara University,
Turkey} \email{bungor@science.ankara.edu.tr}

%\date{}
\newtheorem{thm}{Theorem}[section]
\newtheorem{lem}[thm]{Lemma}
\newtheorem{prop}[thm]{Proposition}
\newtheorem{cor}[thm]{Corollary}
\newtheorem{exs}[thm]{Examples}
\newtheorem{defn}[thm]{Definition}
\newtheorem{nota}{Notation}
\newtheorem{rem}[thm]{Remark}
\newtheorem{ex}[thm]{Example}
\newtheorem{que}[thm]{Question}

\begin{abstract} In this paper, we study on the primeness and semiprimeness
of a Morita context related to the rings and modules. Necessary
and sufficient conditions are investigated for an ideal of a
Morita context to be a prime ideal and a semiprime ideal. In
particular, we determine the conditions under which a Morita
context is prime and semiprime. \vspace{2mm}


\end{abstract}

\maketitle

\section{Introduction} Throughout this paper all rings are associative with $1\neq 0$
and modules are unitary. A proper ideal $I$ of a ring $R$ is
called \emph{prime} if for any elements $a$ and $b$ in $R$,
$aRb\subseteq I$ implies $a\in I$ or $b\in I$, equivalently if for
every pair of ideals $A, B$ of $R$, $AB\subseteq I$ implies
$A\subseteq I$ or $B\subseteq I$. The \emph{prime radical} of a
ring $R$ is the intersection of all prime ideals of $R$ and
denoted by $P(R)$. It is clear that $P(R)$ contains every nilpotent
ideal of $R$. Recall that a proper ideal $I$ of $R$ is said to be
\emph{semiprime} if for any $a\in R$, $aRa\subseteq I$ implies
$a\in I$, equivalently for every ideal $A$ of $R$, $A^2\subseteq
I$ implies $A\subseteq I$. The concepts of prime and semiprime
ideals have very important roles in the ring theory. A ring $R$ is
called \emph{prime} if $0$ is a prime ideal, equivalently for any
$a, b\in R$, $aRb=0$ implies $a=0$ or $b=0$. A ring $R$ is said to
be \emph{semiprime} if it has no nonzero nilpotent ideal,
equivalently for any $a\in R$, $aRa=0$ implies $a=0$. Obviously,
prime rings are semiprime. Also, it is well-known that a ring $R$
is semiprime if and only if $P(R)=0$. A proper submodule $N$ of a
left $R$-module $M$ is called \emph{prime} if for any $r\in R$ and
$m\in M$, $rRm\subseteq N$ implies $rM\subseteq N$ or $m\in N$. It
is easy to show that if $N$ is a prime submodule of $M$, then the
annihilator of the module $M/N$ is a two-sided prime ideal of $R$.
A module $M$ is said to be \emph{prime} if $0$ is a prime
submodule of $M$, i.e., for any $r\in R$ and $m\in M$, $rRm=0$
implies $rM=0$ or $m=0$.

 A \textit{Morita context} is a 4-tuple
$(R,V,W,S)$, where $R,S$ are rings, $_RV_S$ and $_SW_R$ are
bimodules with context products $V\times W\rightarrow R$ and
$W\times V\rightarrow S$ written multiplicatively as $(v,w)\mapsto
vw$ and $(w,v)\mapsto wv$, such
that $\left(%
\begin{array}{cc}
  R& V \\
  W & S \\
\end{array}%
\right)$ is an associative ring with the usual matrix operations.
Morita contexts appeared in the work of Morita \cite{Mor} and they
play an important role in the category of rings and modules that
described equivalences between full categories of modules over
rings. One of the fundamental results in this direction says that
the categories of left modules over the rings $R$ and $S$ are
equivalent if and only if there exists a strict Morita context
connecting $R$ and $S$. This concept appears in disguise in the
literature of the ring theory. A Morita context forms a very large
class of rings generalizing matrix rings. Formal triangular matrix
rings are some of the obvious examples of Morita contexts, and
have been  studied extensively. The concept of Morita context has
been studied by many researchers (see,
e.g.,\cite{Good,Hang,Mcconnel,Tang}). In \cite{Sands}, Sands
investigated radicals of Morita context and obtained valuable
results about general cases of radicals. Recently, in \cite{Tang},
Tang et. al. specify in detail about the structure of ideals in a
Morita context, worthy results are exhibited concerned with the
structures of some special rings $K_s(R)$ and some cases of a
Morita context. $K_s(R)$ is a special kind of the Morita context.
Let $R=V=W=S$ and
$s$ is a central element of $R$. The 4-tuple $\left(%
\begin{array}{cc}
  R & R \\
  R & R \\
\end{array}%
\right)$ becomes a ring with usual matrix addition and with
multiplication defined by $$\left(%
\begin{array}{cc}
  r_1 & v_1 \\
  w_1 & s_1 \\
\end{array}%
\right)\left(%
\begin{array}{cc}
  r_2 & v_2 \\
  w_2 & s_2 \\
\end{array}%
\right)=\left(%
\begin{array}{cc}
  r_1r_2+sv_1w_2 & r_1v_2+v_1s_2 \\
  w_1r_2+s_1w_2 & sw_1v_2+s_1s_2 \\
\end{array}%
\right).$$ This ring is denoted by $K_s(R)$. In \cite{Tang}, it is
proved that $K_s(R)$ is prime if and only if $R$ is prime and
$s\neq 0$, and $K_s(R)$ is semiprime if and only if $R$ is
semiprime and $s$ is a non zero-divisor. Also prime ideals of
$K_s(R)$ are determined. Motivated by these, in this paper, we
study on the primeness and semiprimeness of a Morita context in a
general setting. We determine equivalent conditions for an ideal
to be a prime ideal or a semiprime ideal in a Morita context. In
particular, prime radical of a Morita context is characterized. We
obtain necessary and sufficient conditions for a Morita context to
be a prime ring and a semiprime ring.

In what follows, $\Bbb{Z}$ and $\Bbb{Z}_n$ denote the ring of
integers and the ring of integers modulo $n$ for some positive
integer $n$, respectively.  Also $\left(%
\begin{array}{cc}
  I & V_1 \\
  W_1 & J \\
\end{array}%
\right)$ denotes the subset of a Morita context $\left(%
\begin{array}{cc}
  R & V \\
  W & S \\
\end{array}%
\right)$ consisting of all matrices $\left(%
\begin{array}{cc}
  r & v \\
  w & s \\
\end{array}%
\right)$ with $r\in I$, $s\in J$, $v\in V_1$, $w\in W_1$ where $I,
V_1, W_1, J$ are subsets of $R, V, W, S$, respectively.

\section{Primeness of a Morita Context}

In this section, we work on ideal structures in a Morita context
in terms of primeness and semiprimeness. We exhibit some results
in order to prove our main theorems. In Lemma
\ref{temel_ideal_teo} for a Morita context $(R,V,W,S)$, elements
of $R\oplus W$ and of $V\oplus S$ are written as column vectors,
and elements of $R\oplus V$ and of $W\oplus S$ are written as row
vectors. For unexplained notations and definitions we refer to
\cite{Tang}. The following result is very useful for this paper.

\begin{lem}\cite[Lemma 2.1]{Tang} \label{temel_ideal_teo}
Let $(R,V,W,S)$ be a Morita context and $\emptyset \neq U\subseteq
(R,V,W,S)$. Then the following hold.
\begin{enumerate}
\item $U$ is a right ideal of $(R,V,W,S)$ if and
only if $U=(C_1(U)C_2(U))$ where $C_1(U)$ is a right $R$-submodule
of $R\oplus W$ and $C_2(U)$ is a right $S$-submodule of $V\oplus
S$ with $C_1(U)V\subseteq C_2(U)$ and $C_2(U)W\subseteq C_1(U)$.

\item $U$ is a left ideal of $(R,V,W,S)$ if and only if $U=\left(%
\begin{array}{c}
  R_1(U) \\
  R_2(U) \\
\end{array}%
\right)$ where $R_1(U)$ is a left $R$-submodule of $R\oplus V$ and
$R_2(U)$ is left $S$-submodule of $W\oplus S$ with
$VR_2(U)\subseteq R_1(U)$ and $WR_1(U)\subseteq R_2(U)$.

\item $U$ is an ideal of $(R,V,W,S)$ if and only if
$U=(I,V_1,W_1,J)$ where $I\trianglelefteq R$, $J\trianglelefteq
S$, $V_1 \leqslant~ _RV_S$ and $W_1 \leqslant ~ _SW_R$ with
$V_1W\subseteq I$, $W_1V\subseteq J$, $IV\subseteq V_1$,
$JW\subseteq W_1$, $VW_1\subseteq I$, $WV_1\subseteq J$,
$VJ\subseteq V_1$, $WI\subseteq W_1$.
\end{enumerate}
\end{lem}

\begin{lem}\label{c1u}
Let $T=(R,V,W,S)$ be a Morita context and $U$ a right ideal of
$T$. Then the following hold.
\begin{enumerate}
\item If $\left(%
\begin{array}{c}
  r \\
  w \\
\end{array}%
\right)\in C_1(U)$, then $\left(%
\begin{array}{cc}
  r & 0 \\
  w & 0 \\
\end{array}%
\right)\in U$,
\item If $\left(%
\begin{array}{c}
  v\\
  s \\
\end{array}%
\right)\in C_2(U)$, then $\left(%
\begin{array}{cc}
  0 & v \\
  0 & s \\
\end{array}%
\right)\in U$.
\end{enumerate}
\end{lem}

\begin{proof} $(1)$ Let  $\left(%
\begin{array}{c}
  r \\
  w \\
\end{array}%
\right)\in C_1(U)$. By \cite{Tang},  there exists $\left(%
\begin{array}{cc}
  r & v \\
  w & s \\
\end{array}%
\right)\in U$. For $\left(%
\begin{array}{cc}
  1 & 0 \\
  0 & 0 \\
\end{array}%
\right)\in T$, as $U$ is a right ideal of $T$,  we have $\left(%
\begin{array}{cc}
  r & v \\
  w & s \\
\end{array}%
\right)\left(%
\begin{array}{cc}
  1 & 0 \\
  0 & 0 \\
\end{array}%
\right)=\left(%
\begin{array}{cc}
  r & 0 \\
  w & 0 \\
\end{array}%
\right)\in U$. The proof of  (2) is similar to that of (1).
\end{proof}

We now investigate primeness for a one sided ideal of a Morita
context.

\begin{prop}\label{sagprime}
Let $(R,V,W,S)$ be a Morita context and $U$ a right ideal of
$(R,V,W,S)$. If $U$ is a prime right ideal of $(R,V,W,S)$, then
$C_1(U)$ and $C_2(U)$ are prime submodules of right $R$-module
$R\oplus W$ and right $S$-module $V\oplus S$, respectively.
\end{prop}

\begin{proof} In order to prove $C_1(U)$ is a prime submodule of
$(R\oplus W)_R$, let $r\in R$ and $\left(%
\begin{array}{c}
  m_1 \\
  m_2 \\
\end{array}%
\right)\in R\oplus W$ with $\left(%
\begin{array}{c}
  m_1 \\
  m_2 \\
\end{array}%
\right)Rr\subseteq C_1(U)$. We show that $(R\oplus W)r\subseteq
C_1(U)$ or $\left(%
\begin{array}{c}
  m_1 \\
  m_2 \\
\end{array}%
\right)\in C_1(U)$. Let $\left(%
\begin{array}{c}
  a \\
  w \\
\end{array}%
\right)\in R\oplus W$. Then for any $\left(%
\begin{array}{cc}
  x & y \\
  z & t \\
\end{array}%
\right)\in (R,V,W,S)$, being $\left(%
\begin{array}{c}
  m_1 \\
  m_2 \\
\end{array}%
\right)(xa+yw)r\in C_1(U)$ and Lemma \ref{c1u} imply that
 $\left(%
\begin{array}{cc}
  m_1 & 0 \\
  m_2 & 0 \\
\end{array}%
\right)\left(%
\begin{array}{cc}
  x & y \\
  z & t \\
\end{array}%
\right)\left(%
\begin{array}{cc}
  ar&0 \\
  wr&0 \\
\end{array}%
\right)=\left(%
\begin{array}{cc}
  m_1(xa+yw)r&0\\
  m_2(xa+yw)r&0\\
\end{array}%
\right)\in U$. Since $U$ is a prime right ideal of $(R,V,W,S)$, $\left(%
\begin{array}{cc}
  m_1 & 0 \\
  m_2 & 0 \\
\end{array}%
\right)\in U$ or $\left(%
\begin{array}{cc}
  ar & 0 \\
  wr & 0 \\
\end{array}%
\right)\in U$. Hence $\left(%
\begin{array}{c}
  m_1 \\
  m_2 \\
\end{array}%
\right)\in C_1(U)$ or $\left(%
\begin{array}{c}
  a \\
  w \\
\end{array}%
\right)r\in C_1(U)$,  as asserted. Similarly, it can be shown that
$C_2(U)$ is a prime right $S$-submodule of $V\oplus S$.
\end{proof}

Proposition \ref{sagprime} is useful to determine a right ideal in
a Morita context is prime or not.

\begin{ex}
Consider the Morita context $T=(\Bbb Z_8,\Bbb Z_8,\Bbb Z_8,\Bbb
Z_8)$, where the context products are the same as the product in
$\Bbb Z_8$. Take $U=(\overline4 \Bbb Z_8,\overline4 \Bbb Z_8, \Bbb
Z_8,\Bbb Z_8)$ as a right ideal of $T$. Note that $\overline 4\Bbb
Z_8\oplus \Bbb Z_8=C_1(U)$. For $\overline 2\in \Bbb
Z_8$ and $\left(%
\begin{array}{c}
  \overline 2\\
  \overline 2 \\
\end{array}%
\right)\in \Bbb Z_8\oplus \Bbb Z_8$, we have $\overline 2 \Bbb Z_8
\left(%
\begin{array}{c}
  \overline 2\\
  \overline 2 \\
\end{array}%
\right)\subseteq  C_1(U)$. Also,
$\left(%
\begin{array}{c}
  \overline 2\\
  \overline 2 \\
\end{array}%
\right)\notin C_1(U)$ and $\overline 2(\Bbb Z_8\oplus \Bbb
Z_8)\nsubseteq C_1(U)$. Hence $C_1(U)$ is not a prime submodule of
$R\oplus W$. Therefore $U$ is not a right prime ideal of $T$.
\end{ex}

\begin{lem}\label{kume_esitligi}
Let $(R,V,W,S)$ be a Morita context. If $(I,V_1,W_1, J)$ is a
semiprime ideal of $(R,V,W,S)$, then we have
\begin{enumerate}
\item $\{v\in V ~:~ vW\subseteq I\}=\{v\in V~:~Wv\subseteq J\}$,
\item $\{w\in W~:~Vw\subseteq I\}=\{w\in W~:~wV\subseteq J\}$.
\end{enumerate}
\end{lem}

\begin{proof} (1) Let $A=\{v\in V ~:~ vW\subseteq I\}$ and $B=\{v\in V~:~Wv\subseteq
J\}$. Assume  $v\in A$. Then $vW\subseteq I$ and so for any $w\in W$ the product  $\left(%
\begin{array}{cc}
  vw & 0 \\
  0 & wv \\
\end{array}%
\right)\left(%
\begin{array}{cc}
  R & V \\
  W & S \\
\end{array}%
\right)\left(%
\begin{array}{cc}
  vw & 0 \\
  0 & wv \\
\end{array}%
\right)=\left(%
\begin{array}{cc}
  vwRvw & vwVwv \\
  wvWvw & wvSwv \\
\end{array}%
\right)$ is contained by $\left(%
\begin{array}{cc}
  I & V_1 \\
  W_1 & J \\
\end{array}%
\right).$ Since $\left(%
\begin{array}{cc}
  I & V_1 \\
  W_1 & J \\
\end{array}%
\right)$ is a semiprime ideal of the Morita context, $\left(%
\begin{array}{cc}
  vw & 0 \\
  0 & wv \\
\end{array}%
\right)\in \left(%
\begin{array}{cc}
  I & V_1 \\
  W_1 & J \\
\end{array}%
\right)$. Hence $Wv\subseteq J$ and so $v\in B$. Thus $A\subseteq
B$. Similarly, it can be shown that $B\subseteq A$. $(2)$ can be
proved via similar techniques.
\end{proof}

\begin{prop}\label{prime submodule} Let $(R,V,W,S)$ be a Morita context, $I$ and $J$ are
prime ideals of $R$ and $S$, respectively. Then the following
hold.\begin{enumerate}
    \item $\{v\in V: vW\subseteq I\}$ is a prime left $R$-submodule
    of $V$.
    \item $\{v\in V: Wv\subseteq J\}$ is a prime right $S$-submodule
    of $V$.
    \item $\{w\in W: Vw\subseteq I\}$ is a prime right $R$-submodule
    of $W$.
    \item $\{w\in W: wV\subseteq J\}$ is a prime left $S$-submodule
    of $W$.
\end{enumerate}
\end{prop}
\begin{proof} (1) Let $A$ denote the left $R$-submodule $\{v\in V: vW\subseteq
I\}$ of $V$, and $r\in R$, $v\in V$ with $rRv\subseteq A$. Then $rRvW\subseteq I$.
Since $I$ is a prime ideal of $I$, $r\in I$ or $vW\subseteq I$. If $r\in I$, then
$rVW\subseteq rR\subseteq I$. Hence $rV\subseteq A$ or $v\in A$. Therefore $A$ is prime. \\
(2) Let $B$ denote the right $S$-submodule $\{v\in V: Wv\subseteq
J\}$ of $V$, and $s\in S$, $v\in V$ with $vSs\subseteq B$. Then
$WvSs\subseteq J$. The ideal $J$ being prime implies that $s\in J$
or $Wv\subseteq J$. If $s\in J$, then
$WVs\subseteq Ss\subseteq J$. Thus $Vs\subseteq B$ or $v\in B$. Therefore $B$ is prime.\\
(3) and (4) are proved similarly.
\end{proof}

 We now give the following result related to an ideal of
a Morita context to be prime.

\begin{thm}\label{prime_temel}
Let $T=(R,V,W,S)$ be a Morita context and consider the following
conditions for an ideal $(I,V_1,W_1,J)$ of $T$.
\begin{enumerate}
\item $(I,V_1,W_1,J)$ is prime.
\item $I$, $J$ are prime ideals of $R$ and $S$, respectively, and $V_1=\{v\in V : vW\subseteq I\}=\{v\in V : Wv\subseteq J\}$ and
 $W_1=\{w\in W : Vw\subseteq I\}=\{w\in W : wV\subseteq
J\}$.
\end{enumerate}
Then {\rm(1)} $\Rightarrow$ {\rm(2)}. If  $VW=R$ and $WV=S$, then
{\rm(2)} $\Rightarrow$ {\rm(1)}.
\end{thm}

\begin{proof}
(1) $\Rightarrow$ (2) Assume that $(I,V_1,W_1,J)$ is a prime ideal
of $T$. In order to see $I$ is a prime ideal of $R$, let $r_1, r_2\in R$ with $r_1Rr_2\subseteq I$. We have  $\left(%
\begin{array}{cc}
  r_1 & 0 \\
  0 & 0 \\
\end{array}%
\right)\left(%
\begin{array}{cc}
  R & V \\
  W & S \\
\end{array}%
\right)\left(%
\begin{array}{cc}
  r_2 & 0 \\
  0 & 0 \\
\end{array}%
\right)=\left(%
\begin{array}{cc}
  r_1Rr_2 & 0 \\
  0 & 0 \\
\end{array}%
\right)$. So $\left(%
\begin{array}{cc}
  r_1Rr_2 & 0 \\
  0 & 0 \\
\end{array}%
\right)\subseteq \left(%
\begin{array}{cc}
  I & V_1 \\
  W_1 & J \\
\end{array}%
\right)$. Since $(I,V_1,W_1,J)$ is a prime ideal, $\left(%
\begin{array}{cc}
  r_1 & 0 \\
  0 & 0 \\
\end{array}%
\right)\in \left(%
\begin{array}{cc}
  I & V_1 \\
  W_1 & J \\
\end{array}%
\right)$ or $\left(%
\begin{array}{cc}
  r_2 & 0 \\
  0 & 0 \\
\end{array}%
\right)\in \left(%
\begin{array}{cc}
  I & V_1 \\
  W_1 & J\\
\end{array}%
\right)$. Then $r_1\in I$ or $r_2\in I$. Hence $I$ is a prime
ideal of $R$. Similarly, it can be shown that $J$ is a prime ideal
of $S$. Let $A=\{v\in V~:~vW\subseteq I\}$. By Lemma
\ref{temel_ideal_teo}, it is known that $V_1\subseteq A=\{v\in
V~:~vW\subseteq I\}$. It is enough to show that $ A\subseteq V_1$. If $x\in A$, then $xW\subseteq I$ and so $$\left(%
\begin{array}{cc}
  0 & x \\
  0 & 0 \\
\end{array}%
\right)\left(%
\begin{array}{cc}
  R & V \\
  W & S \\
\end{array}%
\right)\left(%
\begin{array}{cc}
  0 & x \\
  0 & 0 \\
\end{array}%
\right)=\left(%
\begin{array}{cc}
  0 & xWx \\
  0 & 0 \\
\end{array}%
\right)\subseteq \left(%
\begin{array}{cc}
  I & V_1 \\
  W_1 & J \\
\end{array}%
\right).$$ Since $(I,V_1,W_1,J)$ is prime, $\left(%
\begin{array}{cc}
  0 & x \\
  0 & 0 \\
\end{array}%
\right)\in \left(%
\begin{array}{cc}
  I & V_1 \\
  W_1 & J \\
\end{array}%
\right)$. So $x\in V_1$. By Lemma \ref{kume_esitligi}, $V_1=\{v\in
V~:~vW\subseteq I\}=\{v\in V~:~Wv\subseteq J\}$. Similarly, one
can show that $W_1=\{w\in W~:~Vw\subseteq I\}=\{w\in
W~:~wV\subseteq J\}$.\\
(2) $\Rightarrow$ (1) Let $\left(%
\begin{array}{cc}
  r_1 & v_1 \\
  w_1 & s_1 \\
\end{array}%
\right), \left(%
\begin{array}{cc}
  r_2 & v_2 \\
  w_2 & s_2 \\
\end{array}%
\right)\in T$ such that $$\left(%
\begin{array}{cc}
  r_1 & v_1 \\
  w_1 & s_1 \\
\end{array}%
\right)\left(%
\begin{array}{cc}
  R & V \\
  W & S \\
\end{array}%
\right)\left(%
\begin{array}{cc}
  r_2 & v_2 \\
  w_2 & s_2 \\
\end{array}%
\right)\subseteq \left(%
\begin{array}{cc}
  I & V_1 \\
  W_1 & J \\
\end{array}%
\right).$$ Then $$\left(%
\begin{array}{cc}
r_1Rr_2+v_1Wr_2+r_1Vw_2+v_1Sw_2 & r_1Rv_2+v_1Wv_2+r_1Vs_2+v_1Ss_2 \\
w_1Rr_2+s_1Wr_2+w_1Vw_2+s_1Sw_2& w_1Rv_2+s_1Wv_2+w_1Vs_2+s_1Ss_2 \\
\end{array}%
\right)$$ is contained by  $\left(%
\begin{array}{cc}
  I & V_1 \\
  W_1 & J \\
\end{array}%
\right)$. By $(1,1)$ entry of above matrix, $r_1Rr_2\subseteq I$.
As $I$ is prime, $r_1\in I$ or $r_2\in I$. Hence we have the
following cases.

Case I: Suppose that $r_1, r_2\in I$. By $(1,1)$ entry of above
matrix, $v_1Sw_2\subseteq I$, and so $v_1WVw_2\subseteq I$. Being
$v_1WRVw_2R\subseteq I$ implies $v_1WR\subseteq I$ or
$Vw_2R\subseteq I$. Thus $v_1\in V_1$ or $w_2\in W_1$. Therefore
we have the following subcases:\begin{enumerate}
    \item[I(a)] Suppose $v_1\notin V_1$ and $w_2\in W_1$. By $(1,2)$ entry of above
matrix, $v_1Ss_2\subseteq V_1$.  By Proposition \ref{prime
submodule}, the right $S$-submodule $V_1$ being prime implies that
$v_1\in V_1$ or $Vs_2\subseteq V_1$. By supposition, we have only
$Vs_2\subseteq V_1$. Then $WVs_2=Ss_2\subseteq WV_1\subseteq J$,
so $s_2\in J$. On the other hand, again by $(1,2)$ entry of above
matrix, $v_1Wv_2\subseteq V_1$. Since $v_1Wv_2W\subseteq
V_1W\subseteq I$ and $v_1\notin V_1$, $v_2W\subseteq I$ due to the
primeness of $I$. Hence $v_2\in
V_1$. Therefore $\left(%
\begin{array}{cc}
  r_2 & v_2 \\
  w_2 & s_2 \\
\end{array}%
\right)\in \left(%
\begin{array}{cc}
  I & V_1 \\
  W_1 & J \\
\end{array}%
\right)$.
    \item[I(b)] Suppose $v_1\in V_1$ and $w_2\notin W_1$. By $(2,1)$ entry of above
matrix, $s_1Sw_2\subseteq W_1$. The left $S$-submodule $W_1$ being
prime and $w_2\notin W_1$ imply that $s_1W\subseteq W_1$. So
$s_1WV=s_1S\subseteq W_1V\subseteq J$. Hence $s_1\in J$. On the
other hand, again by $(2,1)$ entry of above matrix,
$w_1Vw_2\subseteq W_1$. Then $w_1Vw_2V\subseteq W_1V\subseteq J$
implies $w_1V\subseteq J$ or $w_2V\subseteq J$, and so $w_1\in
W_1$ or $w_2\in W_1$. By supposition, the only possibility is
$w_1\in W_1$. Therefore $\left(%
\begin{array}{cc}
  r_1 & v_1 \\
  w_1 & s_1 \\
\end{array}%
\right)\in \left(%
\begin{array}{cc}
  I & V_1 \\
  W_1 & J \\
\end{array}%
\right)$.
    \item[I(c)] Suppose $v_1\in V_1$ and $w_2\in W_1$. By $(2,1)$ entry of above
matrix, $w_1Rr_2\subseteq W_1$. Since $W_1$ is a prime right
$R$-submodule of $W$, $w_1\in W_1$ or $Wr_2\subseteq W_1$. Being
$Wr_2\subseteq W_1$ implies that $VWr_2\subseteq VW_1\subseteq I$,
and so $r_2\in I$, but we have this fact at the beginning of Case
I. It follows that $w_1\in W_1$ or $w_1\notin W_1$.
\begin{enumerate}
    \item[I(c)-i] Suppose  $w_1\in W_1$. By $(2,2)$ entry of above
matrix, $s_1Wv_2\subseteq J$. Then $Ss_1SWv_2\subseteq J$. The
primeness of $J$ implies that $s_1\in J$ or $Wv_2\subseteq J$, so
$v_2\in V_1$. \begin{itemize}
    \item If $s_1\in J$, then $\left(%
\begin{array}{cc}
  r_1 & v_1 \\
  w_1 & s_1 \\
\end{array}%
\right)\in \left(%
\begin{array}{cc}
  I & V_1 \\
  W_1 & J \\
\end{array}%
\right)$.
    \item If $s_1\notin J$, then $v_2\in V_1$. By $(2,2)$ entry of above
matrix, $s_1Ss_2\subseteq J$. By supposition and the primeness of
$J$, we have $s_2\in J$. Thus $\left(%
\begin{array}{cc}
  r_2 & v_2 \\
  w_2 & s_2 \\
\end{array}%
\right)\in \left(%
\begin{array}{cc}
  I & V_1 \\
  W_1 & J \\
\end{array}%
\right)$.
\end{itemize}
    \item[I(c)-ii] Suppose that $w_1\notin W_1$. By $(2,2)$ entry of above
matrix, $w_1Rv_2=w_1VWv_2\subseteq J$. Since $J$ is prime and
$w_1VSWv_2S\subseteq J$, $w_1V\subseteq J$ or $Wv_2\subseteq J$.
It follows that $w_1\in W_1$ or $v_2\in V_1$. By supposition, we
only have $v_2\in V_1$. On the other hand, again by $(2,2)$ entry
of above matrix, $w_1Vs_2\subseteq J$. Then $w_1VSs_2S\subseteq
J$. The primeness of $J$ implies $w_1V\subseteq J$ or $s_2\in J$.
Since $w_1\notin W_1$, we have  $s_2\in J$. Therefore $\left(%
\begin{array}{cc}
  r_2 & v_2 \\
  w_2 & s_2 \\
\end{array}%
\right)\in \left(%
\begin{array}{cc}
  I & V_1 \\
  W_1 & J \\
\end{array}%
\right)$.
\end{enumerate}
\end{enumerate}

Case II: Without loss of generality, suppose that $r_1\notin I$
and $r_2\in I$. By $(1,1)$ entry of above matrix,
$r_1Vw_2\subseteq I$. Hence $r_1RVw_2R\subseteq I$. Since $I$ is a
prime ideal, $r_1R\subseteq I$ or $Vw_2\subseteq I$. As $r_1\notin
I$, $Vw_2\subseteq I$. So $w_2\in W_1$. By $(1,2)$ entry of above
matrix, $r_1Rv_2\subseteq V_1$. Since $V_1$ is a prime right
$R$-submodule of $V$, $r_1V\subseteq V_1$ or $v_2\in V_1$. Assume
that $r_1V\subseteq V_1$. Hence $r_1VW=r_1R\subseteq I$. This
contradiction implies that $v_2\in V_1$. Lastly, from $(1,2)$
entry of above matrix, $r_1Vs_2\subseteq V_1$. Thus
$r_1Vs_2W\subseteq I$. Then $r_1RVs_2W\subseteq I$. This implies
that $r_1R\subseteq I$ or $Vs_2W\subseteq I$, as I is prime. Since
$r_1R\nsubseteq I$, $Vs_2W\subseteq I$. Thus
$Vs_2WV=Vs_2S\subseteq IV\subseteq V_1$. So $WVs_2S=Ss_2S\subseteq
WV_1\subseteq J$. Then $s_2\in J$ and therefore $\left(%
\begin{array}{cc}
  r_2 & v_2 \\
  w_2 & s_2 \\
\end{array}%
\right)\in \left(%
\begin{array}{cc}
  I & V_1 \\
  W_1 & J \\
\end{array}%
\right)$. This completes the proof.
\end{proof}

 The conditions $VW=R$ and $WV=S$ are not superfluous in Theorem
 \ref{prime_temel} ($(2)\Rightarrow (1)$).

\begin{ex}
 Consider the Morita context $T=\left(%
\begin{array}{cc}
  \Bbb Z_6 & \overline2\Bbb Z_6  \\
  \overline3\Bbb Z_6  & \Bbb Z_6  \\
\end{array}%
\right)$, where the context products are the same as the product in $\Bbb Z_6$.
Then $VW=0$ and $WV=0$. For the ideal $H=\left(%
\begin{array}{cc}
  \overline3\Bbb Z_6  & \overline2\Bbb Z_6  \\
  \overline3\Bbb Z_6  & \overline2\Bbb Z_6  \\
\end{array}%
\right)$ of $T$, the entries of $H$ satisfy all
conditions in Theorem \ref{prime_temel} (2). Although $$\left(%
\begin{array}{cc}
  \overline3 & \overline0 \\
  \overline0 & \overline3 \\
\end{array}%
\right)\left(%
\begin{array}{cc}
  \Bbb Z_6 & \overline2\Bbb Z_6  \\
  \overline3\Bbb Z_6  & \Bbb Z_6  \\
\end{array}%
\right)\left(%
\begin{array}{cc}
  \overline 1 & \overline0 \\
  \overline 0& \overline 2 \\
\end{array}%
\right)=\left(%
\begin{array}{cc}
  \overline3\Bbb Z_6 & 0 \\
  \overline3\Bbb Z_6 & 0 \\
\end{array}%
\right)\subseteq H,$$ $\left(%
\begin{array}{cc}
  \overline3 & \overline0 \\
  \overline0 & \overline3 \\
\end{array}%
\right)\notin H$ and $\left(%
\begin{array}{cc}
  \overline1 & \overline0 \\
  \overline0 & \overline2 \\
\end{array}%
\right)\notin H$. Therefore $H$ is not a prime ideal of $T$.
\end{ex}

 The following theorem is proved in \cite[Theorem
1]{Sands} by Sands. We give a short proof by using prime ideals
determined in Theorem \ref{prime_temel}.
\begin{thm}\label{prime-radical}
Let $T=(R,V,W,S)$ be a Morita context and $P(T)$ the prime radical
of
 $T$. Then $P(T)=\left(%
\begin{array}{cc}
  P(R) & V_0 \\
  W_0 & P(S) \\
\end{array}%
\right)$, where $V_0=\{v\in V~:~vW\subseteq P(R)\}=\{v\in
V~:~Wv\subseteq P(S)\}$ and $W_0=\{w\in W~:~wV\subseteq
P(S)\}=\{w\in W~:~Vw\subseteq P(R)\}$.
\end{thm}

\begin{proof}
Let $\mathcal {K}_1$ and $\mathcal K_2$ be prime spectrums of $R$
and $S$, respectively. Let $I\in \mathcal K_1 $ and $J\in \mathcal
K_2$. Also, $V_I=\{v\in V~:~vW\subseteq I\}=\{v\in V~:~Wv\subseteq
J\}$ and $W_J=\{w\in W~:~wV\subseteq J\}=\{w\in W~:~Vw\subseteq
I\}$. Then $(I,V_I,W_J,J)$ is a prime ideal of $T$.  Let $U$ be a
prime ideal of $T$. Then $U$ has the form $(I, V_I, W_J, J)$  for
some $I\in \mathcal K_1 $, $J\in \mathcal K_2 $ by Lemma
\ref{temel_ideal_teo}, Lemma \ref{kume_esitligi} and Theorem
\ref{prime_temel}. If all of the ideals $(I, V_I, W_J, J)$ for
$I\in \mathcal K_1 $, $J\in \mathcal K_2$ are intersected, then
the result is obtained by an easy calculation.
\end{proof}

%\begin{rem} One may suspect that Theorem \ref{prime-radical} is not similar to
%\cite[Theorem 1]{Sands}, because of our assumptions $VW=R$ and
%$WV=S$. But when the proof of Theorem \ref{prime_temel} is read
%carefully, the assumptions $VW=R$ and $WV=S$ are used only the
%proof of $(2) \Rightarrow (1)$ in Theorem \ref{prime_temel}. To
%remove the complexity of expression Theorem \ref{prime_temel}, we
%write the assumptions $VW=R$ and $WV=S$ on the top.
%\end{rem}

Note that $V/V_0$ is a left $R/P(R)$ right $S/P(S)$-bimodule and
$W/W_0$ is a left $S/P(S)$ right $R/P(R)$-bimodule. Hence
$(R/P(R),V/V_0,W/W_0,S/P(S))$ is a Morita context where the
context products are given by
$$(v+V_0)(w+W_0)=vw+P(R),~ (w+W_0)(v+V_0)=wv+P(S)$$ for all $v\in
V$ and $w\in W$. Then we have the next result.

\begin{prop}
Let $T=(R,V,W,S)$ be a Morita context, and \linebreak $\left(%
\begin{array}{cc}
  R/P(R) & V/V_0 \\
  W/W_0 & S/P(S) \\
\end{array}%
\right)$ be defined above. Then $$T/P(T)\cong \left(%
\begin{array}{cc}
  R/P(R) & V/V_0 \\
  W/W_0 & S/P(S) \\
\end{array}%
\right).$$
\end{prop}

\begin{proof}
Let $f$ denote the map from $T/P(T)$ to $\left(%
\begin{array}{cc}
  R/P(R) & V/V_0 \\
  W/W_0 & S/P(S) \\
\end{array}%
\right)$ defined by $\left(%
\begin{array}{cc}
  r & v \\
  w & s \\
\end{array}%
\right)+P(T) \mapsto \left(%
\begin{array}{cc}
  r+P(R) & v+V_0 \\
  w+W_0 & s+P(S) \\
\end{array}%
\right)$. By easy calculations, it can be shown that $f$ is a ring
isomorphism.
\end{proof}

 Now we give a result related to an ideal of a Morita context to
be a semiprime ideal.

\begin{thm}\label{s.prime-temel}
Let $T=(R,V,W,S)$ be a Morita context. Then the following are
equivalent for an ideal $(I,V_1,W_1,J)$ of $T$.
\begin{enumerate}
\item $(I,V_1,W_1,J)$ is semiprime.
\item $I$ and $J$ are semiprime ideals of $R$ and $S$, respectively, and $V_1=\{v\in V : vW\subseteq I\}=\{v\in V : Wv\subseteq J\}$
and $W_1=\{w\in W : Vw\subseteq I\}=\{w\in W : wV\subseteq J\}$.
\end{enumerate}
\end{thm}

\begin{proof}
(1) $\Rightarrow$ (2) In order to see $I$ is a semiprime ideal of $R$, let $r\in R$ such that $rRr\subseteq I$. Then we have  $\left(%
\begin{array}{cc}
  r & 0 \\
  0 & 0 \\
\end{array}%
\right)\left(%
\begin{array}{cc}
  R & V \\
  W & S \\
\end{array}%
\right)\left(%
\begin{array}{cc}
  r & 0 \\
  0 & 0 \\
\end{array}%
\right)=\left(%
\begin{array}{cc}
  rRr & 0 \\
  0 & 0 \\
\end{array}%
\right)$ $\subseteq \left(%
\begin{array}{cc}
  I & V_1 \\
  W_1 & J \\
\end{array}%
\right)$. Since $\left(%
\begin{array}{cc}
  I & V_1 \\
  W_1 & J \\
\end{array}%
\right)$ is a semiprime ideal of $T$, $\left(%
\begin{array}{cc}
  r & 0 \\
  0 & 0 \\
\end{array}%
\right)\in \left(%
\begin{array}{cc}
  I & V_1 \\
  W_1 & J \\
\end{array}%
\right)$. Hence $r\in I$, and so $I$ is a semiprime ideal of $R$.
Similarly, one can show that $J$ is a semiprime ideal of $S$. By
Lemma \ref{temel_ideal_teo}, it is known that $V_1\subseteq
A=\{v\in V~:~vW\subseteq I\}$. It is enough to show $A\subseteq
V_1$. Take $x\in A$. So
$xW\subseteq I$. Hence $\left(%
\begin{array}{cc}
  0 & x \\
  0 & 0 \\
\end{array}%
\right)\left(%
\begin{array}{cc}
  R & V \\
  W & S \\
\end{array}%
\right)\left(%
\begin{array}{cc}
  0 & x \\
  0 & 0 \\
\end{array}%
\right)=\left(%
\begin{array}{cc}
  0 & xWx \\
  0 & 0 \\
\end{array}%
\right)\subseteq  \left(%
\begin{array}{cc}
  I & V_1 \\
  W_1 & J \\
\end{array}%
\right)$. The ideal  $\left(%
\begin{array}{cc}
  I & V_1 \\
  W_1 & J \\
\end{array}%
\right)$ being semiprime implies that $\left(%
\begin{array}{cc}
  0 & x \\
  0 & 0 \\
\end{array}%
\right)\in \left(%
\begin{array}{cc}
  I & V_1\\
  W_1 & J \\
\end{array}%
\right)$ and so $x\in V_1$. Thus $A\subseteq V_1$. By Lemma
\ref{kume_esitligi}, $V_1=\{v\in V~:~vW\subseteq I\}=\{v\in
V~:~Wv\subseteq J\}$. Similarly, one can show that $W_1=\{w\in
W~:~Vw\subseteq I\}=\{w\in W~:~wV\subseteq J\}$.\\
(2) $\Rightarrow$ (1) Let $\left(%
\begin{array}{cc}
  r & v \\
  w & s \\
\end{array}%
\right)\in T$ with $\left(%
\begin{array}{cc}
  r & v \\
  w & s \\
\end{array}%
\right)\left(%
\begin{array}{cc}
  R & V \\
  W & S \\
\end{array}%
\right)\left(%
\begin{array}{cc}
  r & v \\
  w & s \\
\end{array}%
\right)=$\\$\left(%
\begin{array}{cc}
  rRr+vWr+rVw+vSw & rRv+vWv+rVs+vSs \\
  wRr+sWr+wVw+sSw & wRv+sWv+wVs+sSs \\
\end{array}%
\right)$ $\subseteq \left(%
\begin{array}{cc}
  I & V_1 \\
  W_1 & J \\
\end{array}%
\right)$. By $(1,1)$ entry of above matrix, $rRr\subseteq I$. As
$I$ is semiprime, $r\in I$. Similarly $s\in J$. By $(2,1)$ entry
of above matrix $wVw\subseteq W_1$. Then $wVwV=(wV)^2\subseteq
W_1V\subseteq J$. Since $J$ is semiprime, $wV\subseteq J$. So
$w\in W_1$. Also by $(1,2)$ entry of above matrix, $vWv\subseteq
V_1$. Then $vWvW=(vW)^2\subseteq V_1W\subseteq I$. Since $I$ is a
semiprime ideal of $R$, $vW\subseteq I$. Thus $v\in V_1$.
Therefore $\left(%
\begin{array}{cc}
  r & v \\
  w & s \\
\end{array}%
\right)\in \left(%
\begin{array}{cc}
  I & V_1 \\
  W_1 & J \\
\end{array}%
\right)$. This completes the proof.
\end{proof}

In Theorem \ref{s.prime-temel}(2), semiprimeness of the ideals $I$
and $J$ in the rings $R$ and $S$, respectively, are not
superfluous as the following example shows.

\begin{ex}
Consider
$T=\left(%
\begin{array}{cc}
  \Bbb Z_4 & \overline 2\Bbb Z_4  \\
  \overline2\Bbb Z_4  & \Bbb Z_4  \\
\end{array}%
\right)$ as a Morita context, where the context products are the
same as the product in $\Bbb Z_4$. It is clear that $H=\left(%
\begin{array}{cc}
  \overline2\Bbb Z_4  & \overline2\Bbb Z_4  \\
  \overline2\Bbb Z_4  & \{\overline 0\} \\
\end{array}%
\right)$ is an ideal of $T$. Also $\left(%
\begin{array}{cc}
  \overline0 & \overline0 \\
  \overline0 & \overline2 \\
\end{array}%
\right)\left(%
\begin{array}{cc}
  \Bbb Z_4 & \overline 2\Bbb Z_4  \\
  \overline2\Bbb Z_4  & \Bbb Z_4  \\
\end{array}%
\right)\left(%
\begin{array}{cc}
  \overline0 & \overline0 \\
  \overline0 & \overline2 \\
\end{array}%
\right)=\left(%
\begin{array}{cc}
  \overline0 & \overline0 \\
  \overline0 & \overline0 \\
\end{array}%
\right)\in H$. But $\left(%
\begin{array}{cc}
  \overline0 & \overline0 \\
  \overline0 & \overline2 \\
\end{array}%
\right)\notin H$. Therefore $H$ is not semiprime ideal of $T$,
also $\{\overline 0\}$ is not semiprime in $\Bbb Z_4$.
\end{ex}

 In the sequel, we determine under what conditions a Morita
context is prime or semiprime.

\begin{thm} \label{prime-context}
Let $(R,V,W,S)$ be a Morita context and consider the following
conditions.
\begin{enumerate}
\item $(R,V,W,S)$ is a prime ring.
\item $R$, $S$ are prime rings and $_RV_S$, $_SW_R$ are prime modules.
\item $R$ and $S$ are prime rings.
\item $R$ or $S$ is a prime ring.
\end{enumerate}
Then {\rm(1)} $\Rightarrow$ {\rm(2)} $\Rightarrow$
{\rm(3)}$\Rightarrow$ {\rm(4)}. If $VW=R$ and $WV=S$, then
{\rm(4)} $\Rightarrow$ {\rm(1)}.
\end{thm}

\begin{proof}
(1) $\Rightarrow$ (2) Let $r_1, r_2\in R$ such that $r_1Rr_2=0$. Then we have $\left(%
\begin{array}{cc}
  r_1 & 0 \\
  0 & 0 \\
\end{array}%
\right)\left(%
\begin{array}{cc}
  R & V \\
  W & S \\
\end{array}%
\right)\left(%
\begin{array}{cc}
  r_2 & 0 \\
  0 & 0 \\
\end{array}%
\right)=$ $\left(%
\begin{array}{cc}
  r_1Rr_2 & 0 \\
  0 & 0 \\
\end{array}%
\right)$. As $(R,V,W,S)$ is prime, $r_1=0$ or $r_2=0$. Hence $R$
is a prime ring. Similarly, it can be shown that $S$ is prime.
Now let $r\in R$ and $v\in V$ with $rRv=0$. Then $\left(%
\begin{array}{cc}
  r & 0 \\
  0 & 0 \\
\end{array}%
\right)\left(%
\begin{array}{cc}
  R & V \\
  W & S \\
\end{array}%
\right)\left(%
\begin{array}{cc}
  0 & v \\
  0 & 0 \\
\end{array}%
\right)=$ $\left(%
\begin{array}{cc}
  0 & rRv \\
  0 & 0 \\
\end{array}%
\right)$. Thus $r=0$ or $v=0$. This implies that $rV=0$ or
$v=0$. Hence $V$ is prime as a left $R$-module. On the other hand, let $s\in S$ and $v\in V$ with $vSs=0$. Since $\left(%
\begin{array}{cc}
  0 & v \\
  0 & 0 \\
\end{array}%
\right)\left(%
\begin{array}{cc}
  R & V \\
  W & S \\
\end{array}%
\right)\left(%
\begin{array}{cc}
  0 & 0 \\
  0 & s \\
\end{array}%
\right)=$ $\left(%
\begin{array}{cc}
  0 & vSs \\
  0 & 0 \\
\end{array}%
\right)$, by (1),  $v=0$ or $s=0$. Thus $v=0$ or $Vs=0$. It follows that $V$ is also prime as a right $S$-module.
By the same technique, it can be shown that $_SW_R$ is a prime module.\\
(2) $\Rightarrow$ (3) and  (3) $\Rightarrow$ (4) are obvious.\\
(4) $\Rightarrow$ (1) Assume that $R$ is prime. Let $VsW=0$ for
$s\in S$. By hypothesis, $WVsWV=SsS=0$. So $s=0$. By \cite[Theorem
27]{Amitsur}, we have $S$ is prime. If $S$ is prime, similar proof
can work for $R$.
 We claim that $0=\left(%
\begin{array}{cc}
  0_R & 0_V \\
  0_W & 0_S \\
\end{array}%
\right)$ is a prime ideal of $\left(%
\begin{array}{cc}
  R & V \\
  W & S \\
\end{array}%
\right)$. Since $R$ and $S$ are prime rings, $0_R$ and $0_S$ are
prime ideals of $R$ and $S$, respectively. Now consider the sets
$\{v\in V : vW=0\}$ and $\{v\in V : Wv=0\}$. Let $v\in \{v\in V :
vW=0\}$. Then $vW=0$. By hypothesis, $vWV=vS=0$, and so $v=0$.
Hence $\{v\in V : vW=0\}=0$. On the other hand, being $VW=R$
implies $\{v\in V : Wv=0\}=0$. Similarly, $\{w\in W :
Vw=0\}=\{w\in W : wV=0\}=0$. Therefore $0$ is a prime ideal of $\left(%
\begin{array}{cc}
  R & V \\
  W & S \\
\end{array}%
\right)$ by Theorem \ref{prime_temel}. This completes the proof.
%Let $\left(%
%\begin{array}{cc}
%  r_1 & v_1 \\
%  w_1 & s_1 \\
%\end{array}%
%\right), \left(%
%\begin{array}{cc}
%  r_2 & v_2 \\
%  w_2 & s_2 \\
%\end{array}%
%\right)\in \left(%
%\begin{array}{cc}
%  R & V \\
%  W & S \\
%\end{array}%
%\right)$ such that  $$\left(%
%\begin{array}{cc}
%  r_1 & v_1 \\
%  w_1 & s_1 \\
%\end{array}%
%\right)\left(%
%\begin{array}{cc}
%  R & V \\
%  W & S \\
%\end{array}%
%\right)\left(%
%\begin{array}{cc}
%  r_2 & v_2 \\
%  w_2 & s_2 \\
%\end{array}%
%\right)= \left(%
%\begin{array}{cc}
%  0 & 0 \\
%  0 & 0 \\
%\end{array}%
%\right).$$ Then the matrix $$\left(%
%\begin{array}{cc}
%  r_1Rr_2+v_1Wr_2+r_1Vw_2+v_1Sw_2 & r_1Rv_2+v_1Wv_2+r_1Vs_2+v_1Ss_2 \\
%  w_1Rr_2+s_1Wr_2+w_1Vw_2+s_1Sw_2 & w_1Rv_2+s_1Wv_2+w_1Vs_2+s_1Ss_2 \\
%\end{array}%
%\right)$$ is zero. By $(1,1)$ entry of above matrix, $r_1Rr_2=0$.
%As $R$ is prime, $r_1=0$ or $r_2=0$. Without loss of generality,
%assume that $r_1\neq 0$ and $r_2=0$. From $(1,1)$ entry of above
%matrix, $r_1Vw_2=0$. Hence $r_1RVw_2=0$. So $r_1R=0$ or $Vw_2=0$.
%Since $r_1\neq 0$, $Vw_2=0$. Then $WVw_2=Sw_2=0$, and so $w_2=0$.
%By $(1,2)$ entry of above matrix, $r_1Rv_2=0$. As $V$ is prime,
%$r_1V=0$ or $v_2=0$. If $r_1V=0$, then $r_1VW=r_1R=0$. As this is
%a contradiction, we conclude that $v_2=0$. Lastly, by $(1,2)$
%entry of above matrix, $r_1Vs_2=0$. So $r_1Vs_2W=0$. Then
%$r_1RVs_2W=0$. As $R$ is prime, $r_1R=0$ or $Vs_2W=0$. We know
%that $r_1R\neq 0$. Hence $Vs_2W=0$. Thus $WVs_2WV=Ss_2S=0$. This
%implies that $s_2=0$.
\end{proof}

\begin{thm}\label{s.prime-ring}
Let $T=(R,V,W,S)$ be a Morita context and consider the following
conditions.
\begin{enumerate}
\item $(R,V,W,S)$ is a semiprime ring.
\item $R$ and $S$ are semiprime rings.
\item $R$ or $S$ is a semiprime ring.
\end{enumerate}
Then {\rm(1)} $\Rightarrow$ {\rm(2)}$\Rightarrow$ {\rm(3)}. If
$VW=R$ and $WV=S$, then {\rm(3)} $\Rightarrow$ {\rm(1)}.
\end{thm}

\begin{proof}
(1) $\Rightarrow$ (2) It is similar to proof of Theorem
\ref{prime-context}  (1) $\Rightarrow$ (2).\\
(2) $\Rightarrow$ (3) It is clear.\\
(3) $\Rightarrow$ (1) Assume that $R$ is a semiprime ring. Let
$VsW=0$ for $s\in S$. Via similar discussion to proof of Theorem
\ref{prime-context}, we have $s=0$. So $S$ is semiprime, by
\cite[Corollary 21]{Amitsur}. Similarly we can prove that $R$ is
semiprime, if $S$ is semiprime. If we show that $P(T)=0$, then the
proof is completed.
By Theorem \ref{prime-radical}, $P(T)=\left(%
\begin{array}{cc}
  P(R) & V_0 \\
  W_0 & P(S) \\
\end{array}%
\right)$ where $V_0=\{v\in V : vW\subseteq P(R)\}=\{v\in V :
Wv\subseteq P(S)\}$ and $W_0=\{w\in W : wV\subseteq P(S)\}=\{w\in
W : Vw\subseteq P(R)\}$. As $R$ and $S$ are semiprime, $P(R)=0$
and $P(S)=0$. We claim that $V_0=0$ and $W_0=0$. Let $v\in V_0$.
Then $vW=0$, and so $vWV=vS=0$. Hence $v=0$, and thus $V_0=0$.
Similarly, one can show that $W_0=0$. Therefore $P(T)=0$, as
desired.
\end{proof}

\end{document}